\documentclass[12pt,reqno]{article}

\usepackage{amssymb}
\usepackage{amsmath}
\usepackage{amsthm}
\usepackage{amsfonts}
\usepackage{amscd}
\usepackage{graphicx}
\usepackage[usenames]{color}
\usepackage{tikz}
\usepackage[colorlinks=true,
linkcolor=webgreen,
filecolor=webbrown,
citecolor=webgreen]{hyperref}

\definecolor{webgreen}{rgb}{0,.5,0}
\definecolor{webbrown}{rgb}{.6,0,0}

\newcommand{\seqnum}[1]{\href{http://oeis.org/#1}{\underline{#1}}}

\theoremstyle{plain}
\newtheorem{theorem}{Theorem}
\newtheorem{lemma}[theorem]{Lemma}
\newtheorem{prop}[theorem]{Proposition}
\newtheorem{cor}[theorem]{Corollary}

\begin{document}

\title{Distributions of Statistics over Pattern-Avoiding Permutations}
\author{Michael Bukata \and Ryan Kulwicki \and Nicholas Lewandowski \and Lara Pudwell \and Jacob Roth \and Teresa Wheeland}

\date{Valparaiso University \\ \today}

\maketitle

\begin{abstract}
We consider the distribution of ascents, descents, peaks, valleys, double ascents, and double descents over permutations avoiding a set of patterns.  Many of these statistics have already been studied over sets of permutations avoiding a single pattern of length 3.  However, the distribution of peaks over 321-avoiding permutations is new and we relate it statistics on Dyck paths.  We also obtain new interpretations of a number of well-known combinatorial sequences by studying these statistics over permutations avoiding two patterns of length 3.
\end{abstract}

\section{Introduction}

Let $\mathcal{S}_n$ denote the set of permutations of $\{1,2,\dots, n\}$ and let $\mathrm{red}(w_1\cdots w_m)$ be the word obtained by replacing the $i$th smallest digit(s) of $w$ with $i$.  Given $\pi \in \mathcal{S}_n$ and $\rho \in \mathcal{S}_m$, we say that $\pi$ \emph{contains} $\rho$ as a pattern if there exist indices $1 \leq i_1 < i_2 < \cdots < i_m \leq n$ such that $\pi_{i_a} < \pi_{i_b}$ if and only if $\rho_a < \rho_b$; that is, $\mathrm{red}(\pi_{i_1}\cdots \pi_{i_m})=\rho$.  Otherwise $\pi$ \emph{avoids} $\rho$.  For example, the permutation $18274635 \in \mathcal{S}_8$ contains the pattern $\rho=4312$ using $i_1=2$, $i_2=4$, $i_3=5$, and $i_4=8$ since the entries of $\pi_{i_1}\pi_{i_2}\pi_{i_3}\pi_{i_4}=8745$ are in the same relative order as 4312; i.e., $\mathrm{red}(8745)=4312$.  Let $\mathcal{S}_n(\rho_1, \dots , \rho_p)$ be the set of permutations avoiding each of $\rho_1, \dots, \rho_p$; $\mathcal{S}_n(\rho_1, \dots , \rho_p)$ is called a \emph{pattern class} and the pattern(s) $\rho_1, \dots, \rho_p$ are called the \emph{basis} of the pattern class. Further, let $\mathrm{s}_n(\rho_1,\dots, \rho_p)=\left|\mathcal{S}_n(\rho_1, \dots, \rho_p)\right|$.   It is well-known that $\mathrm{s}_n(\rho)=\frac{\binom{2n}{n}}{n+1}$ (\seqnum{A000108}) when $\rho \in \mathcal{S}_3$, and there are a variety of techniques for determining $\mathrm{s}_n(\rho_1, \dots, \rho_p)$, depending on that patterns to be avoided.

Another well-known family of objects enumerated by the Catalan numbers (\seqnum{A000108}) is the set of \emph{Dyck paths} of semilength $n$. Here a Dyck path of semilength $n$ is a sequence of $n$ up-steps ($U=\langle 1,1\rangle$) and $n$ down-steps ($D=\langle 1,-1\rangle$) from $(0,0)$ to $(2n,0)$ that never falls below the $x$-axis.  We let $\mathcal{D}_n$ denoted the set of such paths.  Further, we let $\mathcal{I}_n$ be the set of indecomposable Dyck paths of semilength $n$, where a path is indecomposable if it only touches the $x$-axis at $(0,0)$ and at $(2n,0)$. Because both $\mathcal{S}_n(\rho)$ and $\mathcal{D}_n$ have the same enumeration when $\rho \in \mathcal{S}_3$, bijections with Dyck paths are a powerful tool to better understand the structure of these pattern classes.

Some common permutations and constructions require addition notation.  To this end, let $I_m = 1\cdots m$ be the increasing permutation of length $m$ and let $J_m = m(m-1)\cdots 1$ be the decreasing permutation of length $m$.  Further, given permutations $\alpha \in \mathcal{S}_a$ and $\beta \in \mathcal{S}_b$, let $\alpha \oplus \beta \in \mathcal{S}_{a+b}$ denote the direct sum of $\alpha$ and $\beta$ and let $\alpha \ominus \beta\in \mathcal{S}_{a+b}$ denote the skew-sum of $\alpha$ and $\beta$, defined as follows:
$$\alpha \oplus \beta = \begin{cases}
\alpha(i),& 1 \leq i \leq a;\\
a+\beta(i-a),& a+1 \leq i \leq a+b.
\end{cases}$$
$$\alpha \ominus \beta = \begin{cases}
\alpha(i)+b,& 1 \leq i \leq a;\\
\beta(i-a),& a+1 \leq i \leq a+b.
\end{cases}$$

Another thread of research is to consider the distribution of permutation statistics over $\mathcal{S}_n$.  Here, a \emph{permutation statistic} is a function $\mathrm{stat}: \mathcal{S}_n \to \mathbb{Z}^{+} \cup \{0\}$.  Some common statistics include ascents ($\mathrm{asc}$), descents ($\mathrm{des}$), double ascents ($\mathrm{dasc}$), double descents ($\mathrm{ddes}$), peaks ($\mathrm{pk}$), and valleys ($\mathrm{vl}$), which are defined as follows:  

$$\mathrm{asc}(\pi) = \left|\left\{i \middle| \pi_i< \pi_{i+1}\right\}\right|,$$

$$\mathrm{des}(\pi) = \left|\left\{i \middle| \pi_i> \pi_{i+1}\right\}\right|,$$

$$\mathrm{dasc}(\pi) = \left|\left\{i \middle| \pi_i< \pi_{i+1} \text{ and } \pi_{i+1}< \pi_{i+2}\right\}\right|,$$

$$\mathrm{ddes}(\pi) = \left|\left\{i \middle| \pi_i> \pi_{i+1} \text{ and } \pi_{i+1}> \pi_{i+2}\right\}\right|,$$

$$\mathrm{pk}(\pi) = \left|\left\{i \middle| \pi_i< \pi_{i+1} \text{ and } \pi_{i+1}> \pi_{i+2}\right\}\right|,$$

$$\mathrm{vl}(\pi) = \left|\left\{i \middle| \pi_i> \pi_{i+1} \text{ and } \pi_{i+1}< \pi_{i+2}\right\}\right|.$$
It is well-known that $\left|\left\{\pi \in \mathcal{S}_n \middle| \mathrm{asc}(\pi) = k\right\}\right|=\left|\left\{\pi \in \mathcal{S}_n \middle| \mathrm{des}(\pi) = k\right\}\right|$ is given by the Eulerian numbers (\seqnum{A008292}), while the distributions of $\mathrm{dasc}$, $\mathrm{ddes}$, $\mathrm{pk}$, and $\mathrm{vl}$, are newer to the literature or open.

Combining these two areas, we consider the distribution of permutation statistics over $\mathcal{S}_n(\rho_1, \dots, \rho_p)$.  Let 
$$\mathrm{a}_{n,k}^{\mathrm{stat}}(\rho_1, \dots, \rho_p) = \left|\left\{\pi \in \mathcal{S}_n(\rho_1, \dots, \rho_p) \middle| \mathrm{stat}(\pi)=k\right\}\right|.$$  Further, for $\pi \in \mathcal{S}_n$, let $\pi^r = \pi_n \cdots \pi_1$ and $\pi^c = (n+1-\pi_1)\cdots (n+1-\pi_n)$ denote the reverse and complement of $\pi$ respectively.  By symmetry, we observe the following:

\begin{align*}
\mathrm{a}_{n,k}^{\mathrm{asc}}(\rho_1, \dots, \rho_p)&=\mathrm{a}_{n,k}^{\mathrm{des}}(\rho_1^r, \dots, \rho_p^r)\\
&=\mathrm{a}_{n,k}^{\mathrm{des}}(\rho_1^c, \dots, \rho_p^c)\\
&=\mathrm{a}_{n,k}^{\mathrm{asc}}(\rho_1^{rc}, \dots, \rho_p^{rc}),
\end{align*}

\begin{align*}
\mathrm{a}_{n,k}^{\mathrm{dasc}}(\rho_1, \dots, \rho_p)&=\mathrm{a}_{n,k}^{\mathrm{ddes}}(\rho_1^r, \dots, \rho_p^r)\\
&=\mathrm{a}_{n,k}^{\mathrm{ddes}}(\rho_1^c, \dots, \rho_p^c)\\
&=\mathrm{a}_{n,k}^{\mathrm{dasc}}(\rho_1^{rc}, \dots, \rho_p^{rc}),
\end{align*}

\begin{align*}
\mathrm{a}_{n,k}^{\mathrm{pk}}(\rho_1, \dots, \rho_p)&=\mathrm{a}_{n,k}^{\mathrm{pk}}(\rho_1^r, \dots, \rho_p^r)\\
&=\mathrm{a}_{n,k}^{\mathrm{vl}}(\rho_1^c, \dots, \rho_p^c)\\
&=\mathrm{a}_{n,k}^{\mathrm{vl}}(\rho_1^{rc}, \dots, \rho_p^{rc}).
\end{align*}

In this paper, we consider $\mathrm{a}_{n,k}^{\mathrm{stat}}(\rho_1, \dots, \rho_p)$ where $p \in \left\{1, 2\right\}$ and where $\mathrm{stat} \in \left\{\mathrm{asc}, \mathrm{des}, \mathrm{dasc}, \mathrm{ddes}, \mathrm{pk}, \mathrm{vl}\right\}$.  A summary our results is given in Table \ref{T:allresults}.  In Section \ref{S:history}, we detail known results for $\mathrm{a}_{n,k}^{\mathrm{stat}}(\rho)$ where $\rho \in \mathcal{S}_3$.  While there are a number of previous results, $\mathrm{a}_{n,k}^{\mathrm{pk}}(321)$ is new, and we determine its distribution in Section \ref{S:peaks} via a bijection with Dyck paths.  In Section \ref{S:patternsets} we consider $\mathrm{a}_{n,k}^{\mathrm{stat}}(\rho_1, \rho_2)$ for $\rho_1, \rho_2 \in \mathcal{S}_3$; while these enumerations yield a number of well-known combinatorial sequences, the particular interpretations in terms of permutation statistics are new.

\begin{table}[hbt]
\begin{center}
\resizebox{\textwidth}{!}{
\begin{tabular}{|c|c|c|c|c|c|c|}
\hline
$B \backslash$ $\mathrm{st}$&$\mathrm{asc}$&$\mathrm{des}$&$\mathrm{pk}$&$\mathrm{vl}$&$\mathrm{dasc}$&$\mathrm{ddes}$\\
\hline
231&(known)&(known)&Thm. \ref{T:pk231}&Thms. \ref{T:biject312and321} and \ref{T:pk321}&(known)&(known)\\
\hline
321&(known)&(known)&Thm. \ref{T:pk321}&Thm. \ref{T:pk321}&(known)&(known)\\
\hline
213,312&Prop. \ref{P:213312asc}&Prop. \ref{P:213312asc}&Prop. \ref{P:213312dasc}&Prop. \ref{P:213312dasc}&Prop. \ref{P:213312pk}&Pr. \ref{P:213312vl}\\
\hline
132,213&Prop. \ref{P:132213asc}&Prop. \ref{P:132213asc}&Prop. \ref{P:132213dasc}&Prop. \ref{P:132213dasc}&Prop. \ref{pk132213}&Prop. \ref{pk132213}\\
\hline
213,231&Prop. \ref{P:132213asc}&Prop. \ref{P:132213asc}&Prop. \ref{P:132213dasc}&Prop. \ref{P:132213dasc}&Prop. \ref{pk132213}&Prop. \ref{pk132213}\\
\hline
123,132&Prop. \ref{asc123132}&Prop. \ref{des123132}&Prop. \ref{dasc123132}&Prop. \ref{P:123132ddes}&Prop. \ref{P:123132pk}&Prop. \ref{vl123132}\\
\hline
132,321&Prop. \ref{asc132321}&Prop. \ref{asc132321}&Prop. \ref{dasc132321}&Prop. \ref{ddes132321}&Prop. \ref{pk132321}&Prop. \ref{vl132321}\\
\hline
\end{tabular}}
\end{center}
\caption{Results for distribution of statistics over $\mathcal{S}_n(B)$}
\label{T:allresults}
\end{table}

\section{History}\label{S:history}

The study of permutation statistics has a rich history, with over 300 possible statistics listed in the database FindStat \cite{FindStat} as of this writing.  However, the distribution of statistics over pattern classes, rather than over all permutations, is newer.   Robertson, Saracino, and Zeilberger \cite{RSZ03} and Mansour and Robertson \cite{MR03} studied the distribution of fixed points over pattern classes whose basis is a subset of $\mathcal{S}_3$. Elizalde \cite{E04} gave an alternate approach to the distribution of fixed points using bijections with Dyck paths and also determined the distribution of excedances over the same pattern classes.

Dokos, Dwyer, Johnson, Sagan, and Selsor \cite{DDJSS12} defined two pattern sets $\{\rho_1, \dots, \rho_p\}$ and $\{\rho_1^\prime, \dots, \rho_p^\prime\}$ to be $\mathrm{st}$-Wilf equivalent if $\mathrm{a}_{n,k}^{\mathrm{st}}(\rho_1, \dots, \rho_p)=\mathrm{a}_{n,k}^{\mathrm{st}}(\rho_1^\prime, \dots, \rho_p^\prime)$ for all $n$ and $k$ and determined all $\mathrm{st}$-Wilf equivalences for subsets of $\mathcal{S}_3$ when $\mathrm{st}$ is the number of inversions or the major index.

Fixed points and excedances are statistics involving a single digit of $\pi$ at a time, while inversions and major index involve multiple digits.  The statistics we study in this paper may best be thought of as consecutive patterns in $\pi$.  In particular, $\mathrm{asc}(\pi)$ is the number of consecutive 12 patterns in $\pi$, $\mathrm{des}(\pi)$ is the number of consecutive 21 patterns in $\pi$, $\mathrm{dasc}(\pi)$ is the number of consecutive 123 patterns in $\pi$ and $\mathrm{ddes}(\pi)$ is the number of consecutive 321 patterns in $\pi$.  Meanwhile, $\mathrm{pk}(\pi)$ is the number of consecutive 132 patterns plus the number of consecutive 231 patterns in $\pi$ and $\mathrm{vl}(\pi)$ is the number of consecutive 213 patterns plus the number of consecutive 312 patterns in $\pi$.  In the following subsections, we review the history of results involving these statistics over specific pattern classes.

\subsection{Ascents and Descents}

Studying ascents and descents over $\mathcal{S}_n(\rho)$ where $\rho \in \mathcal{S}_3$ yields one of exactly two sequences: \seqnum{A001263} (the Narayana numbers) or \seqnum{A091156}.  

For $\rho \in \left\{132, 213, 231, 312\right\}$, $\mathrm{a}_{n,k}^{\mathrm{asc}}(\rho)= \mathrm{a}_{n,k}^{\mathrm{des}}(\rho) = \dfrac{\binom{n-1}{k}\binom{n}{k}}{k+1}$ (\seqnum{A001263}).  This enumeration follows by a bijection between permutations in $\mathcal{S}_n(231)$ with $k$ ascents with Dyck paths of semilength $n$ with $k$ DU factors, which are known to be enumerated by \seqnum{A001263}.  For more details, see Petersen \cite{P15}. 

On the other hand, for $\rho \in \left\{123, 321\right\}$ $\mathrm{a}_{n,k}^{\mathrm{asc}}(\rho)= \mathrm{a}_{n,k}^{\mathrm{des}}(\rho)$ is given by \seqnum{A091156}.  In 2010, Barnabei, Bonetti, and Silimbani \cite{BBS10} showed that $$G(q, z) = \sum_{n \geq 0} \sum_{k \geq 0} \mathrm{a}_{n,k}^{\mathrm{des}}(321) q^{k}z^n$$ satisfies 
$$z(1-z+qz)G^2 - G + 1 = 0.$$ 
Their work features a bijection between 321-avoiding permutations of length $n$ and Dyck paths of semilength $n$.  Tracking descents in the permutations corresponds to tracking both DU and DDD factors in the corresponding Dyck path.  In Section \ref{S:peaks}, we make use of the same bijection to study the distribution of peaks over 321-avoiding permutations.  As a corollary, we obtain a simpler way of tracking descents in permutations via the corresponding Dyck paths.

\subsection{Peaks, Valleys, and More}

Table \ref{T:onepatstats} shows the distributions of $\mathrm{pk}$, $\mathrm{vl}$, $\mathrm{dasc}$, and $\mathrm{ddes}$ over $\mathcal{S}_n(\rho)$ for $\rho \in \mathcal{S}_3$.  Notice that by reversal, understanding the distributions when $\rho \in \left\{231, 312, 321\right\}$ determines the distributions for the remaining patterns. The relationship between the first two rows of the table follows from the fact that $231^{rc}=312$. 

\begin{table}[hbt]
\begin{center}
\begin{tabular}{|c|c|c|c|c|}
\hline
$\rho \backslash$ $\mathrm{st}$&$\mathrm{pk}$&$\mathrm{vl}$&$\mathrm{dasc}$&$\mathrm{ddes}$\\
\hline
231& \seqnum{A091894}& \seqnum{A236406} & \seqnum{A092107} & \seqnum{A092107}\\
\hline
312& \seqnum{A236406}& \seqnum{A091894} & \seqnum{A092107} & \seqnum{A092107}\\
\hline
321& \seqnum{A236406} & \seqnum{A236406} & new & (none)\\
\hline
\end{tabular}
\end{center}
\caption{Distribution of statistics over $\mathcal{S}_n(\rho)$ for $\rho \in \mathcal{S}_3$}
\label{T:onepatstats}
\end{table}

In 2010, Barnabei, Bonetti, and Silimbani \cite{BBS10b} used bijections with Dyck paths to consider the joint distribution of various consecutive patterns with descents over $\mathcal{S}_n(312)$.  In another recent paper, Pan, Qiu, and Remmel \cite{PQR18} also investigated the distribution of consecutive patterns of length 3 over $\mathcal{S}_n(132)$ and $\mathcal{S}_n(123)$.   As seen above, their work directly addresses the distributions of $\mathrm{dasc}$ and $\mathrm{ddes}$.  However, $\mathrm{pk}$ and $\mathrm{vl}$ involve combining the distributions of two of their statistics at a time.  Since the results for $\mathrm{dasc}$ and $\mathrm{ddes}$ are already studied in \cite{BBS10b, PQR18}, we focus on $\mathrm{pk}$ and provide an alternate approach in Section \ref{S:peaks}.  Once we have determined $a_{n,k}^{\mathrm{pk}}(\rho)$ for $\rho \in \{231, 312, 321\}$, by symmetry, we have determined the distributions of $\mathrm{pk}$ and $\mathrm{vl}$ over all $\mathcal{S}_n(\rho)$ for $\rho \in \mathcal{S}_3$.  In Section \ref{S:patternsets} we extend this work to pattern classes that avoid two or more patterns.

\section{Peaks}\label{S:peaks}

We now wish to determine $a_{n,k}^{\mathrm{pk}}(\rho)$ for $\rho \in \{231, 312, 321\}$.

\begin{theorem}\label{T:pk231}
For $n \geq 1$, $k \geq 0$,
$a_{n,k}^{\mathrm{pk}}(231)=\dfrac{2^{n-2k-1}\binom{n-1}{2k}\binom{2k}{k}}{k+1}$.
\end{theorem}

We prove Theorem \ref{T:pk231} via a bijection with Dyck paths. The bijection in this proof is given by Petersen \cite{P15} for the purpose of determining $\mathrm{a}_{n,k}^{\mathrm{des}}(\rho)$ and the eumeration given in Theorem \ref{T:pk231} is given by Petersen in \seqnum{A091894} of the On-line Encyclopedia of Integer Sequences.  We include the argument here for completeness.

\begin{proof}
Define a bijection $\phi: \mathcal{S}_n(231) \to \mathcal{D}_n$ recursively as follows.  The empty permutation maps to the empty path and $\phi(1)=UD$.  Now, for $\pi \in \mathcal{S}_n(231)$ where $n \geq 2$, suppose that $\pi_i=n$ and write $\pi=\pi_1\cdots \pi_{i-1} n \pi_{i+1} \cdots \pi_n$.  Let $\alpha=\mathrm{red}(\pi_1\cdots \pi_{i-1})$ and $\beta = \mathrm{red}(\pi_{i+1} \cdots \pi_n)$.  Notice that $\alpha$ or $\beta$ could be empty.  Now $\phi(\pi)=\phi(\alpha)U\phi(\beta)D$.

Notice that $\pi$ has a peak involving $n$ exactly when $\alpha$ and $\beta$ are both non-empty.  Since $\phi(\alpha)$ ends in a D and $\phi(\beta)$ begins in a U, by construction, there is a peak in $\pi$ involving $n$ exactly when the corresponding $U$ in $\phi(\pi)$ is part of a DUU factor.  Recursively, the number of peaks of $\pi$ corresponds to the number of DUU factors of $\phi(\pi)$, or, equivalently, to the number of DDU factors in the reversal of $\phi(\pi)$.  The number of paths in $\mathcal{D}_n$ with $k$ DDU factors  is given to be $\dfrac{2^{n-2k-1}\binom{n-1}{2k}\binom{2k}{k}}{k+1}$ in OEIS sequence \seqnum{A091894}.  
\end{proof}

The fact that $a_{n,k}^{\mathrm{pk}}(312)=a_{n,k}^{\mathrm{pk}}(321)$ requires another well-known bijection.

\begin{theorem}\label{T:biject312and321}
For all $n$ and $k$, $a_{n,k}^{\mathrm{pk}}(312)=a_{n,k}^{\mathrm{pk}}(321)$.
\end{theorem}

The bijection below is a symmetry of a well-known bijection of Simion and Schmidt \cite{SS85} using left-to-right maxima.  They used this bijection to to show that $\mathrm{s}_n(312) = \mathrm{s}_n(321)$ for all $n$, while we use it to show the refinement that $a_{n,k}^{\mathrm{pk}}(312)=a_{n,k}^{\mathrm{pk}}(321)$.  We say that $\pi_i$ is a \emph{left-to-right maximum} of $\pi$ if $\pi_i>\pi_j$ for $j<i$.  For example the left-to-right maxima of $32658741$ are 3, 6, and 8.

\begin{proof}
We define a bijection $\zeta: \mathcal{S}_n(312) \to \mathcal{S}_n(321)$ that preserves left-to-right maxima.

Consider $\pi \in \mathcal{S}_n(312)$.  Suppose that the left-to-right maxima of $\pi$ are $\ell_1, \dots, \ell_k$ and that they are located in positions $j_1, \dots, j_k$.  We claim that $\pi$ is the unique 312-avoiding permutation with left-to-right maxima $\ell_1, \dots, \ell_k$ in positions $j_1, \dots, j_k$.  In particular, we determine the other entries of $\pi$ from left to right by placing in position $i$ the largest unused digit that is smaller than the rightmost left-to-right maxima before position $i$.

Similarly, there is a unique 321-avoiding permutation with left-to-right maxima $\ell_1, \dots, \ell_k$ in positions $j_1, \dots, j_k$.  In particular, the digits that are not left-to-right maxima must appear in increasing order.  Let $\zeta(\pi)$ be this permutation.

Notice that any peak of $\pi$ must involve a left-to-right maxima as its middle entry, and similarly any peak of $\zeta(\pi)$ must involve a left-to-right maxima as its middle entry.  Since $\zeta$ preserves left-to-right maxima both in value and in position, $\zeta$ preserves peaks.
\end{proof}

Finally, we determine $a_{n,k}^{\mathrm{pk}}(321)$, which is the central result of this paper.  Previously, Baxter \cite{B14} computed data about $a_{n,k}^{\mathrm{pk}}(321)$ using an enumeration scheme algorithm; however our generating function is new and our bijective argument ties together a number of previously-known generating function results.

\begin{theorem} \label{T:pk321}
$$\sum_{n \geq 0} \sum_{k \geq 0} a_{n,k}^{\mathrm{pk}}(321) q^kz^n=1+z\left(-\dfrac{-1+\sqrt{-4z^2q+4z^2-4z+1}}{2z(zq-z+1)}\right)^2.$$
\end{theorem}

We first describe a bijection $\psi: \mathcal{S}_n(321) \to \mathcal{D}_n$ that is due to Krattenthaler \cite{K01}. Consider $\pi \in \mathcal{S}_n(321)$ and plot the points $(i,\pi_i)$ for $1 \leq i \leq n$.  Let $P=\left\{(p_1, \pi_{p_1}), \dots, (p_k, \pi_{p_k})\right\}$ be the set of points $(i, \pi_i)$ such that $\pi_i$ is not a left-to-right maxima of $\pi$ and $p_1<p_2< \cdots <p_k$.  Then define a path of $E=\langle 1,0\rangle$ steps and $N=\langle 0,1 \rangle$ steps from $(1,0)$ to $(n+1,n)$ in the following way: use $p_1-1$ $E$ steps followed by $\pi_{p_1}$ $N$ steps to get from $(1,0)$ to $(p_1,\pi_{p_1})$.  For $1 \leq i \leq k-1$, use $(p_{i+1}-p_i)$ $E$ steps followed by $(\pi_{p_{i+1}}-\pi_{p_i})$ $N$ steps to get from $(p_i,\pi_{p_i})$ to $(p_{i+1},\pi_{p_{i+1}})$.  Finally, take $((n+1)-p_k)$ $E$ steps followed by $(n-\pi_{p_k})$ $N$ steps to get from $(p_k,\pi_{p_k})$ to $(n+1,n)$.  Figure \ref{F:Krateg} shows this process for $\pi=617238459$.  By construction, this path stays below the line $y=x-1$, and we obtain the Dyck path $\psi(\pi)$ by replacing all $E$ steps with $U$ and all $N$ steps with $D$.  Therefore, $\psi(617238459)=UDUUDUDUUDUDUUDDDD$.

\begin{figure}
\begin{center}
\scalebox{0.5}{\begin{tikzpicture}
\draw[white,line width=8pt] (0,0)--(10,0)--(10,10)--(0,10)--(0,0);
\draw[step=1.0,black,thin,opacity=0.5] (0,0) grid (10,10);
\draw[thin,gray,dashed] (1,0)--(10,9);
\draw (0,0) -- (0,10)--(10,10)--(10,0)--(0,0);
\draw[line width=4pt] (1,0) -- (2,0)--(2,1)--(4,1)--(4,2)--(5,2)--(5,3)--(7,3)--(7,4)--(8,4)--(8,5)--(10,5)--(10,9);
\fill[black] (1,6) circle (0.3cm);
\fill[black] (2,1) circle (0.3cm);
\fill[black] (3,7) circle (0.3cm);
\fill[black] (4,2) circle (0.3cm);
\fill[black] (5,3) circle (0.3cm);
\fill[black] (6,8) circle (0.3cm);
\fill[black] (7,4) circle (0.3cm);
\fill[black] (8,5) circle (0.3cm);
\fill[black] (9,9) circle (0.3cm);
\end{tikzpicture}}
\end{center}

\caption{Bijection $\psi$ applied to $\pi=617238459$}
\label{F:Krateg}
\end{figure}
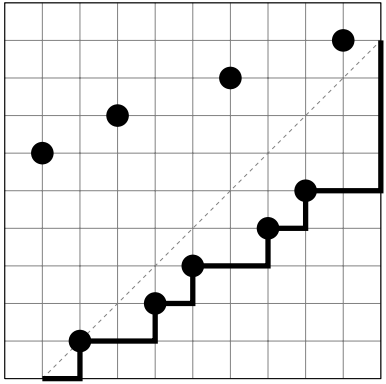

We know that a 321-avoiding permutation can be partitioned into two increasing subsequences: namely, the left-to-right maxima, and the remaining digits.  Necessarily, the middle digit of a peak in such a permutation must be a left-to-right maxima, and the final digit is not.  After $\pi_1$, whenever we have a left-to-right maxima in $\pi$, we have a UU factor in $\psi(\pi)$.  Whenever we have a non-left-to-right maxima in $\pi$, we have at least one D in $\psi(\pi)$. Therefore, a peak of $\pi \in \mathcal{S}_n(321)$ corresponds to a UUD factor in $\psi(\pi)$, with one exception.  A UUD factor that is followed only by Ds indicates that $\pi$ ended with a left-to-right maxima.  To this end, we introduce two statistics on Dyck paths.  Let $\mathrm{st}(d)$ be the number of UUD factors in Dyck path $d$, and let $\mathrm{st^*}(d)$ be the number of UUD factors in Dyck path $d$ that appear before the last $U$.  For example, $\mathrm{st}(UUUDDDUD)  = \mathrm{st^*}(UUUDDDUD) = 1$, while $\mathrm{st}(UUUDDDUUDD)=2$ and $\mathrm{st^*}(UUUDDDUUDD)=1$. We have just seen that
$$\sum_{n \geq 0} \sum_{k \geq 0} a_{n,k}^{\mathrm{pk}}(321) q^kz^n=\sum_{n \geq 0} \sum_{d \in \mathcal{D}_n} q^{\mathrm{st^*}(d)}z^{\left|d\right|}.$$
It remains to study the distribution of $\mathrm{st^*}$ on Dyck paths of semilength $n$.

We define the following four generating functions, which are weight-enumerators on Dyck paths.  Throughout, $\mathrm{st}(d)$ and $\mathrm{st^*}(d)$ are as defined above, $\mathcal{D}_n$ is the set of all Dyck paths of semilength $n$, and $\mathcal{I}_n$ is the set of indecomposable Dyck paths of semilength $n$.

\begin{center}
\begin{tabular}{cc}
$A:=\displaystyle{\sum_{n \geq 0} \sum_{d \in \mathcal{D}_n} q^{\mathrm{st}(d)}z^{n}},$&$B:=\displaystyle{\sum_{n \geq 0} \sum_{d \in \mathcal{I}_n} q^{\mathrm{st}(d)}z^{n}},$\\
$C:=\displaystyle{\sum_{n \geq 0} \sum_{d \in \mathcal{D}_n} q^{\mathrm{st^*}(d)}z^{n}},$&$D:=\displaystyle{\sum_{n \geq 0} \sum_{d \in \mathcal{I}_n} q^{\mathrm{st^*}(d)}z^{n}}.$
\end{tabular}
\end{center}

Notice that our goal is to find $C(q,z)$.  By construction we have $C = 1 + AD$ and $A = 1 + AB$.  We prove Theorem \ref{T:pk321} by first determining $A$ and $D$.

\begin{lemma} \label{L:D}
$$D(q,z)=\sum_{n \geq 0} \sum_{k \geq 0} a_{n,k}^{\mathrm{des}}(321) q^kz^{n+1}.$$
\end{lemma}

\begin{proof}
Suppose $\pi \in \mathcal{S}_n(321)$.  Any descent in $\pi$ consists of a left-to-right maxima followed by a non-left-to-right-maxima.  Using bijection $\psi$, defined above, $\pi$ has a left-to-right maxima at the beginning of $\pi$ and also whenever $\psi(\pi)$ has a UU factor.  Similarly, $\pi$ has a non-left-to-right maxima whenever it has a UD factor, unless the D is at the end of $\psi(\pi)$.  Together, we detect a descent in $\pi$ when $\psi(\pi)$ begins with a UD factor and whenever $\psi(\pi)$ has a UUD factor before the last $U$.  To convert the first case into a UUD factor, let $\widehat{\psi}(\pi)$ be the Dyck path obtained by adding a U to the beginning and a D to the end of $\psi(\pi)$.  By construction, $\widehat{\psi}(\pi)$ is an indecomposable Dyck path of semilength $n+1$.  Now, each descent in $\pi$ corresponds to a UUD factor in $\widehat{\psi}(\pi)$ that appears before the final U, which proves the lemma.
\end{proof}

Next, we consider $A(q,z)$.

\begin{lemma} \label{L:A}
$$A(q,z)=\sum_{n \geq 0} \sum_{k \geq 0} a_{n,k}^{\mathrm{des}}(321) q^kz^{n}.$$
\end{lemma}

\begin{proof}
By definition, $A(q,z)$ tracks all UUD factors across Dyck paths of semilength $n$.  

We have seen in the proof of Lemma \ref{L:D} that there is at most one UUD factor in $\psi(\pi)$ that does not correspond to a descent of $\pi$, namely a UUD factor that is followed only by Ds.  Similarly, there is at most one descent in $\pi$ that does not correspond to a UUD factor in $\psi(\pi)$, namely a descent at the beginning of $\pi$ corresponds to $\psi(\pi)$ beginning with a UD factor.  In other words, given $\pi \in \mathcal{S}_n(321)$ with $d=\psi(\pi)$, either $\mathrm{st}(d) = \mathrm{des}(\pi)$, $\mathrm{st}(d)+1 = \mathrm{des}(\pi)$, or $\mathrm{st}(d) = \mathrm{des}(\pi)+1$,

We prove the lemma by giving an involution $\iota$ on $\mathcal{D}_n$.  

If $d=\psi(\pi)$ with $\mathrm{st}(d) = \mathrm{des}(\pi)$, then $\iota(d)=d$.  

Now, consider $d=\psi(\pi)$ with $\mathrm{st}(d)=k$ and $\mathrm{des}(\pi)=k+1$. Since $\pi$ has one more descent than $\mathrm{st}(d)$, we know that $\psi(\pi)$ begins with UD and $d$ does not have a UUD factor at the end.  In other words, $d=(UD)^id^{\prime}DUD^j$ for some positive $i$ and $j$ where $d^{\prime}$ is a sequence of $n-i-1$ Us and $n-i-j-1$ Ds that does not begin in $UD$.  Let $\iota(d)=d^{\prime}DU(U^iD^i)D^j$.  Now, by construction, $\iota(d)$ has $k+1$ UUD factors, since a new UUD factor was introduced at the end, but $\mathrm{des}(\phi^{-1}(\iota(d)))=k$ since there is no longer an initial UD in $\iota(d)$.

Finally, consider $d=\psi(\pi)$ with $\mathrm{st}(d)=k+1$ and $\mathrm{des}(\pi)=k$. Since $d$ has one more UUD factor than $\mathrm{des}(\pi)$, we know that $d$ ends with $DU^iD^j$ for some $j \geq i \geq 2$ and $d$ does not begin with UD.  In other words, $d=d^{\prime}DU^iD^j$ where $d^{\prime}$ is a sequence of $n-i$ Us and $n-j-1$ Ds that does not begin in $UD$.  Let $\iota(d)=(UD)^{i-1}d^{\prime}DUD^{j-i}$.  Now, by construction, $\iota(d)$ has $k$ UUD factors, since a UUD factor was removed at the end, but $\mathrm{des}(\phi^{-1}(\iota(d)))=k+1$ since there is a new initial UD in $\iota(d)$.

By involution $\iota$, we see that UUD factors on Dyck paths are equidistributed with descents in 321-avoiding permutations.
\end{proof}

An example of $\iota$ in action is shown in Figure \ref{F:involution}.

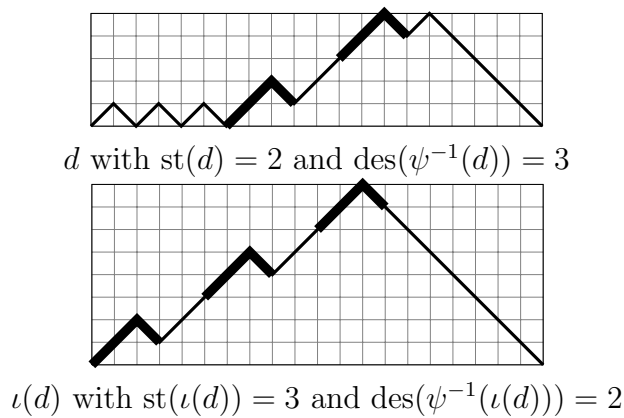
\begin{figure}
\begin{center}
\scalebox{0.3}{\begin{tikzpicture}
\draw[white,line width=8pt] (0,0)--(20,0)--(20,5)--(0,5)--(0,0);
\draw[step=1.0,black,thin,opacity=0.5] (0,0) grid (20,5);
\draw (0,0) -- (20,0)--(20,5)--(0,5)--(0,0);
\draw[line width=4pt] (0,0) -- (1,1)--(2,0)--(3,1)--(4,0)--(5,1)--(6,0)--(8,2)--(9,1)--(13,5)--(14,4)--(15,5)--(20,0);
\draw[line width=12pt] (6,0)--(8,2)--(9,1);
\draw[line width=12pt] (11,3)--(13,5)--(14,4);
\end{tikzpicture}}

$d$ with $\mathrm{st}(d)=2$ and $\mathrm{des}(\psi^{-1}(d)) = 3$

\scalebox{0.3}{\begin{tikzpicture}
\draw[white,line width=8pt] (0,0)--(20,0)--(20,8)--(0,8)--(0,0);
\draw[step=1.0,black,thin,opacity=0.5] (0,0) grid (20,8);
\draw (0,0) -- (20,0)--(20,8)--(0,8)--(0,0);
\draw[line width=4pt] (0,0)--(2,2)--(3,1)--(7,5)--(8,4)--(9,5)--(12,8)--(20,0);
\draw[line width=12pt] (0,0)--(2,2)--(3,1);
\draw[line width=12pt] (5,3)--(7,5)--(8,4);
\draw[line width=12pt] (10,6) -- (12,8)--(13,7);
\end{tikzpicture}}

$\iota(d)$ with $\mathrm{st}(\iota(d))=3$ and $\mathrm{des}(\psi^{-1}(\iota(d))) = 2$

\end{center}

\caption{An example of $\iota(d)$}
\label{F:involution}
\end{figure}

As a consequence of Lemmas \ref{L:D} and \ref{L:A}, we see that $D=zA$.  Therefore, $C(q,z)=1+AD = 1+zA^2$.  Using Barnabei, Bonetti, and Silimbani's result for $A(q,z)$ in \cite{BBS10} yields Theorem \ref{T:pk321}.

Two nice observations follow from this proof.  First, Barnabei, Bonetti, and Silimbani determined $A(q,z)$ by counting the number of DU factors plus the number of DDU factors in $\phi(\pi)$.  We have shown in Lemma \ref{L:A} that $A(q,z)$ can be determined by counting only the number of UUD factors in $\phi(\pi)$.  Second, by Lemma \ref{L:A} and the fact that $A=1+AB$, we can determine $B(q,z)$.  It turns out

$$B(q,z)=z(1-q)+\sum_{n \geq 0} \sum_{k \geq 0} a_{n,k}^{\mathrm{pk}}(231) q^{k+1}z^{n+1},$$
matching the enumeration in Theorem \ref{T:pk231}.

Thus, the distributions of both $\mathrm{st}$ and $\mathrm{st^*}$ are in bijection with distributions of statistics over pattern-avoiding permutations whether we consider them over all Dyck paths or only over indecomposable Dyck paths.

\section{Avoiding Two Patterns}\label{S:patternsets}

We now consider $\mathrm{a}_{n,k}^{\mathrm{stat}}(\rho_1, \rho_2)$ where $\mathrm{stat} \in \left\{\mathrm{asc}, \mathrm{des}, \mathrm{dasc}, \mathrm{ddes}, \mathrm{pk}, \mathrm{vl}\right\}$ and $\rho_1,\rho_2 \in \mathcal{S}_3$.  Using the symmetries of reverse and complement, there are 6 pairs of patterns to consider: $\{123,321\}$, $\{213, 312\}$, $\{132, 213\}$, $\{213, 231\}$, $\{123, 132\}$, and $\{132, 321\}$.  Simion and Schmidt \cite{SS85} determined $\left|\mathcal{S}_n(\rho_1,\rho_2)\right|$ for each of these classes.  We now use the permutation structures they determined to find $\mathrm{a}_{n,k}^{\mathrm{stat}}(\rho_1, \rho_2)$ for our desired statistics.  We already know that $\left|\mathcal{S}_n(123,321)\right|=0$ for $n \geq 5$, so there are 5 non-trivial pairs of permutation patterns to consider.  A summary of the results of this section is given in Tables \ref{T:pairs1}, \ref{T:pairs2}, and \ref{T:pairs3}.  Just as many results for $\mathrm{a}_{n,k}^{\mathrm{stat}}(\rho)$ with $\rho \in \mathcal{S}_3$ follow from bijections with Dyck paths, many results in this section follow from bijections with binary sequences.

\begin{table}
\begin{center}

\begin{tabular}{| c | c | c |}
		\hline
		Patterns \textbackslash Statistic & $\mathrm{asc}$ & $\mathrm{des}$\\ \hline
	
			213,312 &
		$\binom{n-1}{k}$ &
		$\binom{n-1}{k}$\\
		\hline
	
		132,213 &
		$\binom{n-1}{k}$ &
		$ \binom{n-1}{k}$ \\ \hline

		213,231 &
		$\binom{n-1}{k}$ &
		$\binom{n-1}{k}$  \\ \hline

		123,132 &
		$\binom{n}{2k}$ &	
		$\binom{n}{2(n-k-1)}$  \\ \hline
		
				132,321 & 
		$\begin{aligned} &1, && k=n-1; \\ &\binom{n}{2}, && k=n-2.  \end{aligned}$ &
		$\begin{aligned} &1, && k=0; \\ &\binom{n}{2}, && k=1.  \end{aligned}$  \\ \hline
	\end{tabular}
\end{center}
\caption{Distribution of $\mathrm{asc}$ and $\mathrm{des}$ over pattern classes of the form $\mathcal{S}_n(\rho_1, \rho_2)$ with $\rho_1, \rho_2 \in \mathcal{S}_3$}
\label{T:pairs1}
\end{table}
	
	\begin{table}
	\begin{center}
	\begin{tabular}{| c | c | c |}
		\hline
		Patterns \textbackslash Statistic & $\mathrm{dasc}$ & $\mathrm{ddes}$ \\ \hline
	
			213,312 &
		$\begin{aligned} & n, && k=0; \\ &\binom{n-1}{k+1}, && k \geq 1. \end{aligned}$ &
		$\begin{aligned} &n, && k=0; \\ &\binom{n-1}{k+1}, && k \geq 1. \end{aligned}$ \\
		\hline
	
		132,213 &
		\seqnum{A076791} &
		\seqnum{A076791}   \\ \hline

		213,231 &
		\seqnum{A076791}  &
		\seqnum{A076791} \\ \hline

		123,132 &
		trivial &
		$\binom{n-2}{k}+2\binom{n-3}{k}$\\ \hline
		
				132,321 & 
		$\begin{aligned} & 1, && k = n-2; \\ & n, && k = n -3; \\ &\binom{n}{2}-n, && k = n -4. \end{aligned}$ &
		trivial \\ \hline
	\end{tabular}
	\end{center}
	
	\caption{Distribution of $\mathrm{dasc}$ and $\mathrm{ddes}$ over pattern classes of the form $\mathcal{S}_n(\rho_1, \rho_2)$ with $\rho_1, \rho_2 \in \mathcal{S}_3$}
\label{T:pairs2}
\end{table}

\begin{table}
\begin{center}
	\begin{tabular}{| c | c | c  |}
		\hline
		Patterns \textbackslash Statistic & $\mathrm{pk}$ & $\mathrm{vl}$ \\ \hline
	
			213,312 &
		$\begin{aligned} & 2, && k=0; \\ & 2^{n-1}-2, && k=1. \end{aligned}$ & trivial \\
		\hline
	
		132,213 &
		$\binom{n}{2k+1}$ &
		$\binom{n}{2k+1}$ \\ \hline

		213,231 &
		$\binom{n}{2k+1}$ &
		$\binom{n}{2k+1}$ \\ \hline

		123,132 &
		$\binom{n}{2k+1}$ &
		$2\cdot\binom{n-1}{2k}$ \\ \hline
		
				132,321 & 
		$\begin{aligned} & n, && k=0; \\ &\binom{n-1}{2}, && k=1. \end{aligned}$ &
		$\begin{aligned}& 2, && k=0; \\ &\binom{n}{2}-1, && k=1. \end{aligned}$ \\ \hline
	\end{tabular}
	\end{center}
\caption{Distribution of $\mathrm{pk}$ and $\mathrm{vl}$ over pattern classes of the form $\mathcal{S}_n(\rho_1, \rho_2)$ with $\rho_1, \rho_2 \in \mathcal{S}_3$}
\label{T:pairs3}
\end{table}

We consider each pattern pair in turn.

\subsection{Statistics on \texorpdfstring{$\mathcal{S}_n(213,312)$}{Sn(213,312)}}

We first describe the structure of a $\{213, 312\}$-avoiding permutation.  Let $\pi \in \mathcal{S}_n(213,312)$.  Suppose that $\pi_i=n$.  Then $\pi_1 \cdots \pi_{i-1}$ must form an increasing subpermutation (otherwise $\pi$ has a 213 pattern), and $\pi_{i+1}\cdots \pi_n$ must form a decreasing subpermutation (otherwise $\pi$ has a 312 pattern).   There are $\binom{n-1}{i-1}$ ways to choose the digits before $\pi_i=n$, so summing over all possible values for $i$, we have that $\left|\mathcal{S}_n(213,312)\right|=\sum_{i=1}^{n} \binom{n-1}{i-1}=2^{n-1}$. This structure helps prove the following propositions.

\begin{prop}\label{P:213312asc}
$$\mathrm{a}_{n,k}^{\mathrm{asc}}(213,312)=\mathrm{a}_{n,k}^{\mathrm{des}}(213,312)=\binom{n-1}{k}.$$
\end{prop}

\begin{proof}
By the structure above, $\pi \in \mathcal{S}_n(213,312)$ has $k$ ascents if and only if $\pi_{k+1}=n$.  There are $\binom{n-1}{k}$ ways to determine the digits before $\pi_{k+1}$, which uniquely determines $\pi$.

Now, $\pi \in \mathcal{S}_n$ has $k$ descents if and only if $\pi$ has $n-k-1$ ascents.  There are $\binom{n-1}{n-k-1}$ permutations $\pi \in \mathcal{S}_n(213,312)$ with $n-k-1$ ascents, so there are $\binom{n-1}{n-k-1} = \binom{n-1}{k}$ such permutations with $k$ descents.
\end{proof}

Proposition \ref{P:213312asc} gives a new interpretation of Pascal's triangle (\seqnum{A007318}).

\begin{prop}\label{P:213312dasc}
For $n \geq 1$,
$$\mathrm{a}_{n,k}^{\mathrm{dasc}}(213,312)=\mathrm{a}_{n,k}^{\mathrm{ddes}}(213,312) = \begin{cases}
n,&k=0;\\
\binom{n-1}{k+1},& k\geq 1.
\end{cases}$$
\end{prop}

\begin{proof}
Suppose $\pi \in \mathcal{S}_n(213,312)$ has no double ascents.  Then either $\pi_1=n$ or $\pi_2=n$.  In other words, the digit $\pi_1$ determines $\pi$, and there are $n$ choices of $\pi_1$, so we have the first case.

Otherwise, if $k \geq 1$, then $\pi \in \mathcal{S}_n(213,312)$ has $k$ double ascents if and only if $\pi_{k+2}=n$.  There are $\binom{n-1}{k+1}$ ways to determine the digits before $\pi_{k+2}$, which uniquely determines $\pi$.

Since reversing $\pi$ is an involution on $\mathcal{S}_n(213,312)$ that sends double ascents to double descents and vice versa, we get the same enumeration for $\mathrm{a}_{n,k}^{\mathrm{ddes}}(213,312)$.
\end{proof}

While the triangle in Proposition \ref{P:213312dasc} is straightforward to compute, it is new to OEIS and given in \seqnum{A299927}.

\begin{prop}\label{P:213312pk}
$$\mathrm{a}_{n,k}^{\mathrm{pk}}(213,312) = \begin{cases}
2,&k=0;\\
2^{n-1}-2,&k=1;\\
0,&\text{otherwise}.\end{cases}$$
\end{prop}

\begin{proof}
Consider $\pi \in \mathcal{S}_n(213,312)$.  By the structure described above, $\pi$ has at most one peak, and if there is a peak, it must use $n$ as its middle digit.  There are two ways to not have a peak; namely, the increasing permutation where $\pi_n=n$ and the decreasing permutation where $\pi_1=n$.  All other $2^{n-1}-2$ permutation in $\mathcal{S}_n(213,312)$ have one peak.
\end{proof}

\begin{prop}\label{P:213312vl}
$$\mathrm{a}_{n,k}^{\mathrm{vl}}(213,312)\begin{cases}
2^{n-1},&k=0;\\
0,&\text{otherwise}.
\end{cases}$$
\end{prop}

\begin{proof}
A valley is either a 213 pattern or a 312 pattern.  By definition every permutation in $\mathcal{S}_n(213,312)$ has 0 valleys.
\end{proof}

\subsection{Statistics on \texorpdfstring{$\mathcal{S}_n(132,213)$}{Sn(132,213)} and \texorpdfstring{$\mathcal{S}_n(213,231)$}{Sn(213,231)}}

The pattern classes $\mathcal{S}_n(132,213)$ and $\mathcal{S}_n(213,231)$ provide the one non-trivial instance where $\mathrm{a}_{n,k}^{\mathrm{stat}}(\rho_1,\rho_2)=\mathrm{a}_{n,k}^{\mathrm{stat}}(\rho_1^{\prime},\rho_2^{\prime})$ for the statistics of this paper.

We first describe the structure of a $\{132, 213\}$-avoiding permutation.  Suppose $\pi \in \mathcal{S}_n(132,213)$.  Since $\pi$ avoids 213, all digits before $\pi_i=n$ must be in increasing order.  Since $\pi$ avoids 132, all digits before $\pi_i=n$ are larger than all digits after $n$.  These observations imply that if $\pi \in \mathcal{S}_n(132,213)$, then $\pi = I_{i_1} \ominus \cdots \ominus I_{i_m}$ for some positive integers $i_1, \dots, i_m$. In fact, there is a natural bijection $\phi_{132,213}$ between $\mathcal{S}_n(132,213)$ and binary sequences $s=s_1\cdots s_{n-1}$ of length $n-1$; namely, if $s=\phi_{132,213}(\pi)$ then $s_i=1$ when $\pi_i<\pi_{i+1}$ and $s_i=0$ when $\pi_i>\pi_{i+1}$.  This bijection implies $\left|\mathcal{S}_n(132,213)\right| = 2^{n-1}$.

Next, we describe the structure of a $\{213, 231\}$-avoiding permutation.  Suppose $\pi \in \mathcal{S}_n(213,231)$.  Then, for all $i$, either $\pi_i  =\min(\pi_i, \pi_{i+1}, \dots \pi_n)$ or $\pi_i  =\max(\pi_i, \pi_{i+1}, \dots \pi_n)$.  If not, then $\pi_i$ together with $\min(\pi_i, \pi_{i+1}, \dots \pi_n)$ and $\max(\pi_i, \pi_{i+1}, \dots \pi_n)$ form either a 213 pattern or a 231 pattern.  Since there are two choices for each digit of $\pi$ before the last digit, $\left|\mathcal{S}_n(213,231)\right| = 2^{n-1}$.  In fact, there is a natural bijection $\phi_{213,231}$ from $\mathcal{S}_n(213,231)$ to the set of binary sequences $s=s_1\cdots s_{n-1}$ of length $n-1$; namely, $s_i=0$ when $\pi_i  =\max(\pi_i, \pi_{i+1}, \dots \pi_n)$ and $s_i=1$ when $\pi_i  =\min(\pi_i, \pi_{i+1}, \dots \pi_n)$.  

Both bijections $\phi_{132,213}$ and $\phi_{213,231}$ help prove the following propositions.

\begin{prop}\label{P:132213asc}
$$\mathrm{a}_{n,k}^{\mathrm{asc}}(132,213)=\mathrm{a}_{n,k}^{\mathrm{des}}(132,213)=\mathrm{a}_{n,k}^{\mathrm{asc}}(213,231)=\mathrm{a}_{n,k}^{\mathrm{des}}(213,231)=\binom{n-1}{k}.$$
\end{prop}

\begin{proof}
By construction $\pi \in \mathcal{S}_n(132,213)$ has an ascent at $i$ if and only if $s=\phi_{132,213}(\pi)$ has $s_i=1$.  Therefore, $\mathrm{a}_{n,k}^{\mathrm{asc}}(132,213)$ is the number of binary sequences of length $n-1$ with exactly $k$ 1s, which is given by $\binom{n-1}{k}$.  Also, $\mathrm{a}_{n,k}^{\mathrm{des}}(132,213)$ is the number of binary sequences of length $n-1$ with exactly $k$ 0s, which is given by $\binom{n-1}{k}$.

Similarly, $\pi \in \mathcal{S}_n(213, 231)$ has an ascent at $i$ if and only if $s=\phi_{213,231}(\pi)$ has $s_i=1$ and $\pi$ has a descent at $i$ if and only if $s=\phi_{213,231}(\pi)$ has $s_i=0$, so the same enumerations follow. 
\end{proof}

Proposition \ref{P:132213asc} gives a new interpretation of Pascal's triangle (\seqnum{A007318}).

\begin{prop} \label{P:132213dasc}
$$\mathrm{a}_{n,k}^{\mathrm{dasc}}(132,213)=\mathrm{a}_{n,k}^{\mathrm{ddes}}(132,213)=\mathrm{a}_{n,k}^{\mathrm{dasc}}(213,231)=\mathrm{a}_{n,k}^{\mathrm{ddes}}(213,231)$$
and

$$\sum_{n \geq 0} \sum_{k \geq 0} \mathrm{a}_{n,k}^{\mathrm{ddes}}(132,213)q^kz^n=\dfrac{1-qz}{1-z-z^2-qz+qz^2}.$$
\end{prop}

\begin{proof}
By construction $\pi \in \mathcal{S}_n(132,213)$ has a double ascent at $i$ if and only if $s=\phi_{132,213}(\pi)$ has $s_i=s_{i+1}=1$ and $\pi$ has a double descent at $i$ if and only if $s=\phi_{132,213}(\pi)$ has $s_i=s_{i+1}=0$.  Similarly, $\pi \in \mathcal{S}_n(213,231)$ has a double ascent at $i$ if and $s=\phi_{213,231}(\pi)$ has $s_i=s_{i+1}=1$ and a double descent at $i$ if and only if $s=\phi_{213,231}(\pi)$ has $s_i=s_{i+1}=0$.  Therefore $\mathrm{a}_{n,k}^{\mathrm{dasc}}(132,213)=\mathrm{a}_{n,k}^{\mathrm{ddes}}(132,213)=\mathrm{a}_{n,k}^{\mathrm{dasc}}(213,231)=\mathrm{a}_{n,k}^{\mathrm{ddes}}(213,231)$.

While there is not a straightforward closed formula, the number of binary strings with $k$ 00 factors can be determined recursively.

Let $a(n,k)$ be the number of strings of length $n$ with $k$ 00 factors, and then let $a_i(n,k)$ be the number of strings of length $n$ with exactly $k$ 00 factors and that begin with $i$ 0s.  By definition
$$\mathrm{a}_{n,k}^{\mathrm{ddes}}(132,213) = a(n-1,k)=\sum_{i=0}^{n-1} a_i(n-1,k).$$

First, consider the case when $i=0$.  We have $a_0(n,k)=a(n-1,k)$ since $i=0$ implies the string must start with 1.  The remaining $n-1$ digits may be any string of length $n-1$ with $k$ 00 factors.

Now, for $i \geq 1$, we have $a_i(n,k)=a(n-1-i, k-(i-1))$.  This is because the initial $i$ digits of our string are 0.  These 0s account for $i-1$ 00 factors.  The next digit is a 1.  The remaining $n-1-i$ digits may be any binary string of length $n-i$ with $k-(i-1)$ 00 factors.

Together, we have:

\begin{align*}
a(n-1,k)=\sum_{i=0}^{n-1} a_i(n-1,k) &= a(n-2,k) + \sum_{i=1}^{n-1} a(n-2-i, k-(i-1))\\
&= a(n-2,k) + \sum_{i=1}^{k+1} a(n-2-i, k-(i-1)).
\end{align*}

Equivalently:

$$\mathrm{a}_{n,k}^{\mathrm{ddes}}(132,213) = \mathrm{a}_{n-1,k}^{\mathrm{ddes}}(132,213)+\sum_{i=1}^{k+1} a(n-1-i, k-(i-1)).$$

This recurrence implies that 
$$\sum_{n \geq 0} \sum_{k \geq 0} \mathrm{a}_{n,k}^{\mathrm{ddes}}(132,213)q^kz^n=\dfrac{1-qz}{1-z-z^2-qz+qz^2}.$$
\end{proof}

The number of binary sequences with exactly $k$ 00 factors is given in OEIS entry \seqnum{A076791}, and Proposition \ref{P:132213dasc} gives a new permutation statistic interpretation of the sequence.

\begin{prop}\label{pk132213}
$$\mathrm{a}_{n,k}^{\mathrm{pk}}(132,213)=\mathrm{a}_{n,k}^{\mathrm{vl}}(132,213)=\mathrm{a}_{n,k}^{\mathrm{pk}}(213,231)=\mathrm{a}_{n,k}^{\mathrm{vl}}(213,231)=\binom{n}{2k+1}.$$
\end{prop}

\begin{proof}
By construction $\pi \in \mathcal{S}_n(132,213)$ has a peak at $i$ if and only if $s=\phi_{132,213}(\pi)$ has $s_i=1$ and $s_{i+1}=0$ and $\pi$ has a valley at $i$ if and only if $s=\phi_{132,213}(\pi)$ has $s_i=0$ and $s_{i+1}=1$.  By symmetry, $\mathrm{a}_{n,k}^{\mathrm{pk}}(132,213)=\mathrm{a}_{n,k}^{\mathrm{vl}}(132,213)$. Similarly, $\pi \in \mathcal{S}_n(213,231)$ has a peak at $i$ if and only if $s=\phi_{213,231}(\pi)$ has $s_i=1$ and $s_{i+1}=0$ and a valley at $i$ if and only if $s=\phi_{213,231}(\pi)$ has $s_i=0$ and $s_{i+1}=1$.  Therefore, $\mathrm{a}_{n,k}^{\mathrm{pk}}(132,213)=\mathrm{a}_{n,k}^{\mathrm{vl}}(132,213)=\mathrm{a}_{n,k}^{\mathrm{pk}}(213,231)=\mathrm{a}_{n,k}^{\mathrm{vl}}(213,231)$.

Let $a(n,k)$ denote the number of binary sequences of length $n$ with $k$ 10 factors.  We wish to determine $a(n-1,k)$.  

Clearly $a(n,0)=n+1$, since a binary sequence with no 10 factors consists of $i$ 0s followed by $n-i$ 1s, and there are $n+1$ choices for the value of $i$.  On the other hand, a sequence with $k$ 10 factors requires at least $2k$ digits, so if $n < 2k$, then $a(n,k)=0$.  Similarly, $a(2k,k) = 1$ corresponds to the 1 way to have a binary sequence of length $2k$ with $k$ 10 factors, namely $1010\cdots 10$.

Now that we have determined the boundary conditions, suppose that $0 < k < \frac{n-1}{2}$. Now suppose $s$ is a binary sequence of length $n$ with $k$ 10 factors. We call a position $s_i$ a switch if $s_i \neq s_{i+1}$.  In all, a sequence of length $n$ has $n-1$ positions where a switch could occur.

If $s$ starts with 1, the sequence switches from 1 to 0 $k$ times and from 0 to 1 either $k$ times or $k-1$ times, so there are $2k$ or $2k-1$ switches.  In the first case, there are $\binom{n-1}{2k}$ ways to choose the locations of the switches and in the second case there are $\binom{n-1}{2k-1}$ ways to choose the locations of the switches for a total of $\binom{n-1}{2k}+\binom{n-1}{2k-1}=\binom{n}{2k}$ binary sequences of length $n$ with $k$ 10 factors that begin in 1.

If $s$ starts with 0, the sequence switches from 1 to 0 $k$ times and from 0 to 1 either $k$ times or $k+1$ times, so there are $2k$ or $2k+1$ switches.  In the first case, there are $\binom{n-1}{2k}$ ways to choose the locations of the switches and in the second case there are $\binom{n-1}{2k+1}$ ways to choose the locations of the switches for a total of $\binom{n-1}{2k}+\binom{n-1}{2k+1}=\binom{n}{2k+1}$ binary sequences of length $n$ with $k$ 10 factors that begin in 0.

Combining these two cases, we have that $a(n,k)=\binom{n}{2k}+\binom{n}{2k+1} = \binom{n+1}{2k+1}$.  Therefore, 
$$\mathrm{a}_{n,k}^{\mathrm{pk}}(132,213)=a(n-1,k)=\binom{n}{2k+1}.$$
\end{proof}

Proposition \ref{pk132213} gives a new interpretation of OEIS sequence \seqnum{A034867}.

\subsection{Statistics on \texorpdfstring{$\mathcal{S}_n(123,132)$}{Sn(123,132)}}

We first describe the structure of a $\{123,132\}$-avoiding permutation.  For $\pi \in \mathcal{S}_n(123,132)$, either $\pi_{n-1}=1$ or $\pi_n=1$; otherwise, $1$, $\pi_{n-1}$ and $\pi_n$ would form a forbidden pattern.  There is a natural bijection $\phi_{123,132}$ between $\mathcal{S}_n(123,132)$ and binary sequences of length $n-1$ that is described recursively as follows:  $\phi_{123,132}(1)=\epsilon$, the empty string.  Then, for $\pi \in \mathcal{S}_n(123,132)$, 
$$\phi_{123,132}(\pi)=\begin{cases}
\phi_{123,132}(\mathrm{red}(\pi_1\cdots \pi_{n-2}\pi_n))0,&\pi_{n-1}=1;\\
\phi_{123,132}(\mathrm{red}(\pi_1\cdots \pi_{n-1}))1,&\pi_n=1.
\end{cases}$$
For example, $\phi_{123,132}(653241) = 11001$.  We can also read a binary string $s$ of length $n-1$ from left to right to construct the corresponding permutation $\phi^{-1}_{123,132}(s)$.  Namely, begin with $\pi^{(1)}=n$.  Then for $1 \leq i \leq n-1$, if $s_i=0$, then $\pi^{(i+1)}=\pi^{(i)}_1\cdots \pi^{(i)}_{i-1}(n-i)\pi^{(i)}_{i}$, and if $s_i=1$, then $\pi^{(i+1)}=\pi^{(i)}(n-1)$.  $\pi =\phi^{-1}_{123,132}(s) = \pi^{(n)}$. 
Because of the bijection $\phi_{123,132}$ with binary strings, we have $\left|\mathcal{S}_n(123,132)\right| = 2^{n-1}$.  We use this bijection to prove the following propositions.

\begin{prop}\label{asc123132}
$$\mathrm{a}_{n,k}^{\mathrm{asc}}(123,132)=\binom{n}{2k}.$$
\end{prop}

\begin{proof}
Suppose $\pi \in \mathcal{S}_n(123,132)$ has $k$ ascents and consider $s=\phi_{123,132}(\pi)$ and the sequence of partial permutations $\pi^{(1)}, \pi^{(2)}, \dots, \pi^{(n)}$ where $\pi^{(n)}=\pi$.  By construction, $\mathrm{asc}(\pi^{(i+1)}) = \mathrm{asc}(\pi^{(i)})$ or $\mathrm{asc}(\pi^{(i+1)}) = \mathrm{asc}(\pi^{(i)})+1$ for all $i$, so we seek to characterize factors in $s$ that introduce a new ascent in $\pi^{(i+1)}$ compared to $\pi^{(i)}$. 

By construction, $\mathrm{asc}(\pi^{(1)})=0$ and $\pi^{(2)}$ has an ascent if and only if $s_1=0$.  For $i \geq 3$, $\mathrm{asc}(\pi^{(i)}) = \mathrm{asc}(\pi^{(i-1)})+1$ if and only if $s_{i-2}=1$ and $s_{i-1}=0$.

Therefore, in order to determine $\mathrm{a}_{n,k}^{\mathrm{asc}}(123,132)$ we wish to count binary strings of length $n-1$ that either begin with 0 and have $k-1$ 10 factors or that begin with 1 and have $k$ 10 factors.  As before, we call a position $s_i$ a switch if $s_i \neq s_{i+1}$, and in all, a sequence of length $n-1$ has $n-2$ positions where a switch could occur.

In the first case, since $s_1=0$ and there are $k-1$ switches from 1 to 0, there must be either $k-1$ or $k$ switches from 0 to 1 for a total of either $2k-2$ or $2k-1$ switches.  In all there are $\binom{n-2}{2k-2}+\binom{n-2}{2k-1} = \binom{n-1}{2k-1}$ such binary strings.

In the second case, since $s_1=1$ and there are $k$ switches from 1 to 0 there must be either $k-1$ or $k$ switches from 0 to 1 for a total of $2k-1$ or $2k$ switches.  In all there are $\binom{n-2}{2k-1}+\binom{n-2}{2k} = \binom{n-1}{2k}$ such binary strings.

Combining both cases, there are $\binom{n-1}{2k-1}+\binom{n-1}{2k}=\binom{n}{2k}$ permutations of length $n$ that avoid 123 and 132 and have exactly $k$ ascents.
\end{proof}

Proposition \ref{asc123132} gives an alternate interpretation to OEIS \seqnum{A034839}.

\begin{prop} \label{des123132}
$$\mathrm{a}_{n,k}^{\mathrm{des}}(123,132)=\binom{n}{2(n-k-1)}.$$
\end{prop}

\begin{proof}
For any permutation $\pi \in \mathcal{S}_n$, $\mathrm{asc}(\pi)+\mathrm{des}(\pi)=n-1$.  Therefore, a permutation of length $n$ with $k$ descents has $n-k-1$ ascents.  By Proposition \ref{asc123132}, $\mathrm{a}_{n,k}^{\mathrm{des}}(123,132)=\binom{n}{2(n-k-1)}$.
\end{proof}

Proposition \ref{des123132} gives an alternate interpretation to OEIS \seqnum{A109446}, which is a symmetry of OEIS \seqnum{A034839}.

\begin{prop} \label{dasc123132}
For $n \geq 3$
$$\mathrm{a}_{n,k}^{\mathrm{dasc}}(123,132)=\begin{cases}
2^{n-1},&k=0;\\
0,&\text{otherwise}.
\end{cases}$$
\end{prop}

\begin{proof}
Since a consecutive 123 pattern is a double ascent, any permutation that avoids 123 has 0 double ascents.
\end{proof}

\begin{prop} \label{P:123132ddes}
For $n \geq 3$,
$$\mathrm{a}_{n,k}^{\mathrm{ddes}}(123,132)=\binom{n-2}{k}+2\binom{n-3}{k}.$$
\end{prop}

\begin{proof}
For $n \leq 2$, every permutation has 0 double descents, so we focus on the case where $n \geq 3$. Similarly, no permutation has more than $n-2$ double descents, so we focus on $k \leq n-2$.

Suppose $\pi \in \mathcal{S}_n(123,132)$ has $k$ ascents and consider $s=\phi_{123,132}(\pi)$ and the sequence of partial permutations $\pi^{(1)}, \pi^{(2)}, \dots, \pi^{(n)}$ where $\pi^{(n)}=\pi$.  By construction, $\mathrm{ddes}(\pi^{(i+1)}) = \mathrm{ddes}(\pi^{(i)})$ or $\mathrm{ddes}(\pi^{(i+1)}) = \mathrm{ddes}(\pi^{(i)})+1$ for all $i$, so we seek to characterize factors in $s$ that introduce a new double descent in $\pi^{(i+1)}$ compared to $\pi^{(i)}$. 

By construction, $\mathrm{ddes}(\pi^{(1)})=\mathrm{ddes}(\pi^{(2)})=0$ and $\pi^{(3)}$ has a double descent if and only if $s_1=s_2=1$.  For $i \geq 4$, $\mathrm{ddes}(\pi^{(i)}) = \mathrm{ddes}(\pi^{(i-1)})+1$ if and only if $s_{i-2}=s_{i-1}=1$ or $s_{i-2}=s_{i-1}=0$.

Therefore we wish to count the number of binary strings of length $n-1$ that begin with 00 and have $k$ additional 00 or 11 factors plus the number of binary strings of length $n-1$ that do not begin with 00 and have $k$ total 00 or 11 factors.

Now, suppose $k=0$.  By our characterization, there are exactly 3 such permutations.  They correspond to $\phi^{-1}_{123,132}(0101\cdots)$, $\phi^{-1}_{123,132}(1010\cdots)$, and $\phi^{-1}_{123,132}(00101010\cdots)$.  This matches our formula above since $\binom{n-2}{0}+2\binom{n-3}{0}=3$ for $n \geq 3$.

Notice that if $k =n-2$ there is $\binom{n-2}{n-2}+2\binom{n-3}{n-2}=1$ permutation with $n-2$ double descents, namely the strictly decreasing permutation, which corresponds to $\phi^{-1}_{123,132}(11\cdots 1)$.

Now, let $a_{n,k}$ be the number of binary strings of length $n$ with $k$ 00 or 11 factors (other than a possible initial 00).  We wish to determine $a_{n-1,k}$. Suppose $n \geq 4$ and $s =s_1\cdots s_n$ is such a string.  If $s_{n-1}=s_n$ then $s_1\cdots s_{n-1}$ is a string of length n-1 with $k-1$ 00 or 11 factors (other than a possible initial 00).  If $s_{n-1}\neq s_n$, then  $s_1\cdots s_{n-1}$ is a string of length n-1 with $k$ 00 or 11 factors (other than a possible initial 00).  This implies that $a_{n,k}=a_{n-1,k-1}+a_{n-1,k}$.

We now proceed to show that $a_{n,k}=\binom{n-1}{k}+2\binom{n-2}{k}$ by induction.  We have confirmed this formula holds when when $k=0$ and $k=n-2$.  In particular, this implies $a_{n,k}=\binom{n-1}{i}+2\binom{n-2}{i}$ for $0 \leq i \leq n-1$ for the case when $n=2$.

Now, suppose that $a_{n-1,i}=\binom{n-2}{i}+2\binom{n-3}{i}$ for $0 \leq i \leq n-2$.  We know that $a_{n,k}=a_{n-1,k-1}+a_{n-1,k}$.  Therefore:
\begin{align*}
a_{n,k}&=a_{n-1,k-1}+a_{n-1,k}\\
&=\binom{n-2}{k-1}+2\binom{n-3}{k-1}+\binom{n-2}{k}+2\binom{n-3}{k}\\
&=\left(\binom{n-2}{k-1}+\binom{n-2}{k}\right)+2\left(\binom{n-3}{k-1}+\binom{n-3}{k}\right)\\
&=\binom{n-1}{k}+2\binom{n-2}{k},
\end{align*}
which is what we wanted to show.
\end{proof}

Proposition \ref{P:123132ddes} gives a new interpretation of OEIS \seqnum{A093560}.

\begin{prop} \label{P:123132pk}
$$\mathrm{a}_{n,k}^{\mathrm{pk}}(123,132)=\binom{n}{2k+1}.$$
\end{prop}

\begin{proof}
Suppose $\pi \in \mathcal{S}_n(123,132)$ has $k$ ascents and consider $s=\phi_{123,132}(\pi)$ and the sequence of partial permutations $\pi^{(1)}, \pi^{(2)}, \dots, \pi^{(n)}$ where $\pi^{(n)}=\pi$.  By construction, $\mathrm{pk}(\pi^{(i+1)}) = \mathrm{pk}(\pi^{(i)})$ or $\mathrm{pk}(\pi^{(i+1)}) = \mathrm{pk}(\pi^{(i)})+1$ for all $i$, so we seek to characterize factors in $s$ that introduce a new peak in $\pi^{(i+1)}$ compared to $\pi^{(i)}$. 

By construction, $\mathrm{pk}(\pi^{(1)})=\mathrm{pk}(\pi^{(2)})=0$. Also, for $i \geq 3$, $\mathrm{pk}(\pi^{(i)}) = \mathrm{pk}(\pi^{(i-1)})+1$ if and only if $s_{i-2}=0$ and $s_{i-1}=1$.  Therefore, we wish to count the number of binary strings $s$ of length $n-1$ with exactly $k$ 01 factors.  We have two cases.

If $s$ begins with 0 then $s$ switches from 0 to 1 $k$ times and $s$ switches from 1 to 0 either $k-1$ times or $k$ times for a total of $2k-1$ or $2k$ switches.  There are $\binom{n-2}{2k-1}+\binom{n-2}{2k} = \binom{n-1}{2k}$ sequences in this case.

If $s$ begins with a 1 then $s$ switches from 0 to 1 $k$ times and $s$ switches from 1 to 0 either $k$ times or $k+1$ times for a total of $2k$ or $2k+1$ switches.  There are $\binom{n-2}{2k}+\binom{n-2}{2k+1} = \binom{n-1}{2k+1}$ sequences in this case.

Combining both cases yields a total of $\binom{n-1}{2k}+\binom{n-1}{2k+1}=\binom{n}{2k+1}$ binary sequences of length $n-1$ with $k$ 01 factors.  By bijection $\phi_{123,132}$, this implies $\mathrm{a}_{n,k}^{\mathrm{pk}}(123,132)=\binom{n}{2k+1}$.
\end{proof}

Proposition \ref{P:123132pk} gives a new interpretation of OEIS \seqnum{A034867}, which also appeared in Proposition \ref{pk132213}.

\begin{prop} \label{vl123132}
$$\mathrm{a}_{n,k}^{\mathrm{vl}}(123,132)=2\binom{n-1}{2k}.$$
\end{prop}

\begin{proof}
Suppose $\pi \in \mathcal{S}_n(123,132)$ has $k$ ascents and consider $s=\phi_{123,132}(\pi)$ and the sequence of partial permutations $\pi^{(1)}, \pi^{(2)}, \dots, \pi^{(n)}$ where $\pi^{(n)}=\pi$.  By construction, $\mathrm{vl}(\pi^{(i+1)}) = \mathrm{vl}(\pi^{(i)})$ or $\mathrm{vl}(\pi^{(i+1)}) = \mathrm{vl}(\pi^{(i)})+1$ for all $i$, so we seek to characterize factors in $s$ that introduce a new valley in $\pi^{(i+1)}$ compared to $\pi^{(i)}$. 

By construction, $\mathrm{vl}(\pi^{(1)})=\mathrm{vl}(\pi^{(2)})=0$, and $\mathrm{vl}(\pi^{(3)})=1$ if and only if $s_1=0$ and $s_2=0$.  For $i \geq 4$, $\mathrm{vl}(\pi^{(i)}) = \mathrm{vl}(\pi^{(i-1)})+1$ if and only if $s_{i-2}=1$ and $s_{i-1}=0$.  Therefore, we wish to count the number of binary strings $s$ of length $n-1$ that either begin with 00 and have $k-1$ 10 factors or that don't begin with 00 and have $k$ 10 factors.  We consider three cases: $s$ begins with 1, $s$ begins with 01, and $s$ begins with 00.

If $s_1=1$, there are $k$ switches from 1 to 0 and either $k-1$ or $k$ switches from 0 to 1 for a total of $2k-1$ or $2k$ switches.  There are $\binom{n-2}{2k-1}+\binom{n-2}{2k} = \binom{n-1}{2k}$ such sequences of length $n-1$.

If $s_1=0$ and $s_2=1$ there are still $k$ switches from 1 to 0 and either $k-1$ or $k$ switches from 0 to 1 after $s_2$ for a total of $2k-1$ or $2k$ switches after $s_2$.  There are $\binom{n-3}{2k-1}+\binom{n-3}{2k} = \binom{n-2}{2k}$ such sequences of length $n-1$.

If $s_1=0$ and $s_2=0$ there are $k-1$ switches from 1 to 0 and either $k-1$ or $k$ switches from 0 to 1 for a total of $2k-2$ or $2k-1$ switches.  There are $\binom{n-3}{2k-2}+\binom{n-3}{2k-1} = \binom{n-2}{2k-1}$ such sequences of length $n-1$.

Combining these cases yields $$\binom{n-1}{2k} + \left(\binom{n-2}{2k}+\binom{n-2}{2k-1}\right) = \binom{n-1}{2k}+\binom{n-1}{2k} = 2\binom{n-1}{2k}$$ such sequences.
\end{proof}

Proposition \ref{vl123132} gives a new interpretation of OEIS \seqnum{A119462}.

\subsection{Statistics on \texorpdfstring{$\mathcal{S}_n(132,321)$}{Sn(132,321)}}

We first describe the structure of a $\{132, 321\}$-avoiding permutation.  

\begin{prop} \label{struct132321}
If $\pi \in \mathcal{S}_n(132,321)$ then $\pi = \left(I_a \ominus I_b\right)\oplus I_{n-a-b}$ for some $1 \leq a \leq n$ and $0 \leq b \leq n-1$.
\end{prop}

\begin{proof} We proceed by induction on $n$; that is, assume that every member of $\mathcal{S}_{n-1}(132,321)$ is of the form $\left(I_a \ominus I_b\right)\oplus I_{(n-1)-a-b}$ and prove this is the case for members of $\mathcal{S}_{n}(132,321)$.  

For the base case, notice that $\mathcal{S}_1(132,321)=\left\{1\right\}$ and $1=I_1$, so the permutation 1 has the desired form where $a=1$ and $b=0$.

For the induction step, suppose $\pi \in \mathcal{S}_{n}(132,321)$.  This implies that \\$\widehat{\pi}=\mathrm{red}(\pi_1\cdots \pi_{n-1}) \in \mathcal{S}_{n-1}(132,321)$. By the induction hypothesis, either $\widehat{\pi}=I_{n-1}$, $\widehat{\pi}=I_a \ominus I_{n-1-a}$ or $\widehat{\pi} = \left(I_a \ominus I_b\right)\oplus I_{(n-1)-a-b}$.

If $\widehat{\pi}=I_{n-1}$, there are two choices for $\pi_n$.  Either $\pi_n=1$, which means $\pi = I_{n-1}\ominus I_1$ or $\pi_n=n$, which means $\pi=I_n$.  Any other choice of $\pi_n$ produces a 132 pattern involving $\pi_n$.

If $\widehat{\pi}=I_a \ominus I_{n-1-a}$, there are two choices for $\pi_n$.  Either $\pi_n = n-a$ which means $\pi = I_a \ominus I_{n-a}$ or $\pi_n=n$ which means $\pi=(I_a \ominus I_{n-1-a})\oplus I_1$.  Any other choice of $\pi_n$ produces a 132 pattern or a 321 pattern involving $\pi_n$.

If $\widehat{\pi} = \left(I_a \ominus I_b\right)\oplus I_{(n-1)-a-b}$, then $\pi_n=n$ which means $\left(I_a \ominus I_b\right)\oplus I_{n-a-b}$.  Any other choice for $\pi_n$ produces a 132 pattern or a 321 pattern.
\end{proof}

As a consequence of Proposition \ref{struct132321}, we have the following Corollary.
\begin{cor}
$\left|\mathcal{S}_n(132,321)\right| = \binom{n}{2}+1.$
\end{cor}

\begin{proof}
The permutation $I_n$ is in $\mathcal{S}_n(132,321)$ for all $n$.

Otherwise, there are $n$ positions in $\pi$.  We may choose one position to be the last digit of $I_a$ and a second position to be the last position of $I_b$.  This choice of two positions uniquely determines the permutation.  There are $\binom{n}{2}$ permutations in $\mathcal{S}_n(132,321) \setminus \left\{I_n\right\}$. 
\end{proof}

The propositions below follow from the structure given in Proposition \ref{struct132321}.

\begin{prop}\label{asc132321}
$$\mathrm{a}_{n,k}^{\mathrm{asc}}(132,321)=\begin{cases}
1,&k=n-1;\\
\binom{n}{2},&k=n-2;\\
0,&\text{otherwise}.
\end{cases}$$
and
$$\mathrm{a}_{n,k}^{\mathrm{des}}(132,321)=\begin{cases}
1,&k=0;\\
\binom{n}{2},&k=1;\\
0,&\text{otherwise}.
\end{cases}$$
\end{prop}

\begin{proof}
We know $\pi = \left(I_a \ominus I_b\right)\oplus I_{n-a-b}$.  If $a=n$, $\pi$ has $n-1$ ascents and 0 descents.  Otherwise, the only descent in $\pi$ is at position $a$, so $\pi$ has $n-2$ ascents and 1 descent.
\end{proof}

\begin{prop}\label{dasc132321}
$$\mathrm{a}_{n,k}^{\mathrm{dasc}}(132,321)=\begin{cases}
1,&k=n-2;\\
n,&k=n-3;\\
\binom{n}{2}-n, &k=n-4;\\
0,&\text{otherwise}.
\end{cases}$$
\end{prop}

\begin{proof}
We know $\pi = \left(I_a \ominus I_b\right)\oplus I_{n-a-b}$.

If $a=n$, $\pi$ has $n-2$ double ascents.  There is one such permutation.

If $a=n-1$ then $b=1$.  This means $\pi$ has $n-3$ double ascents in $I_a$.  There is one such permutation.

If $a=1$ then there are 0 double ascents in $a$ and there are $(n-1)-2$ double ascents in $I_b \oplus I_{n-a-b} = I_{n-1}$ for a total of $n-3$ double ascents.  There are $n-1$ such permutations since there are $n-1$ choices for the value of $b$, i.e., $1 \leq b \leq n-1$.

So far we have accounted for 1 permutation with $n-2$ double ascents and $1+(n-1)=n$ permutations with $n-3$ double ascents.

If $2 \leq a \leq n-2$, then $\pi$ has $a-2$ double ascents in $I_a$ and $(n-a)-2$ double ascents in $I_b \oplus I_{n-a-b} = I_{n-1}$ for a total of $(a-2)+(n-a-2) = n-4$ double ascents.  The remaining $\binom{n}{2} -n$ permutations fall into this category, which completes the proof.
\end{proof}

\begin{prop}\label{ddes132321}
For $n \geq 3$, 
$$\mathrm{a}_{n,k}^{\mathrm{ddes}}(132,321)=\begin{cases}
\binom{n}{2}+1,&k=0;\\
0,&\text{otherwise}.
\end{cases}$$
\end{prop}

\begin{proof}
Since a consecutive 321 pattern is a double descent, any permutation that avoids 321 has 0 double descents.
\end{proof}

\begin{prop}\label{pk132321}
$$\mathrm{a}_{n,k}^{\mathrm{pk}}(132,321)=\begin{cases}
n,&k=0;\\
\binom{n-1}{2},&k=1;\\
0,&\text{otherwise}.
\end{cases}$$
\end{prop}

\begin{proof}
There is at most one peak in a permutation of the form $\left(I_a \ominus I_b\right)\oplus I_{n-a-b}$.  In particular, we get a peak exactly when $2 \leq a \leq n-1$.

There are $n-1$ permutations where $a=1$, and there is 1 permutation where $a=n$, so there are a total of $n$ permutations with 0 peaks.

The remaining $\binom{n}{2}+1 - n = \binom{n-1}{2}$ permutations have one peak.
\end{proof}

\begin{prop}\label{vl132321}
$$\mathrm{a}_{n,k}^{\mathrm{vl}}(132,321)\begin{cases}
2,&k=0;\\
\binom{n}{2}-1,&k=1;\\
0,&\text{otherwise}.
\end{cases}$$
\end{prop}

\begin{proof}
There is at most one valley in a permutation of the form $\left(I_a \ominus I_b\right)\oplus I_{n-a-b}$.  In particular, we get a valley exactly when $2 \leq n-a \leq n-1$.  The only permutations that violate this rule are when $a=n$ and when $a=n-1$.  There is one permutation with $a=n$, i.e., $I_n$.  There is one permutation with $a=n-1$, i.e., $I_{n-1}\ominus I_1$.  All other $\binom{n}{2}-1$ permutations avoiding 132 and 321 have a valley involving the last digit in $I_a$ and the first two digits of $I_b \oplus I_{n-a-b}$.
\end{proof}

\section{Acknowledgments}
This work was partially supported by NSF grant DUE-1068346.

\end{document}